\documentclass[10pt]{amsart}
\usepackage[]{amsmath, amsthm, amsfonts, amssymb}
\usepackage[all]{xy}

\usepackage[colorlinks=true,linkcolor=blue]{hyperref}

 \newcommand {\N} {{\mathbb N}}
 \newcommand {\C} {{\mathbb C}}
 \newcommand {\R} {{\mathbb R}}
 \newcommand {\Z} {{\mathbb Z}}
 \newcommand {\Q} {{\mathbb Q}}

 \newcommand {\calH} {{\mathcal H}}
 \newcommand {\F} {{\mathcal F}}

 \newcommand {\dt} {{\bullet}}
 
\newcommand {\calK} {{\mathcal K}}
 \newcommand {\G} {{\mathcal G}}

 \newcommand {\Gm} {\mathbb{G}_m}

\newcommand {\LL} {{\mathbb L}}

\newcommand {\A} {{\mathcal N}}

\newcommand {\M} {{\mathcal M}^{\text{eff}}}
\newcommand {\cM} {\mathcal {M}}

\newcommand {\cH} {\mathcal{H}}
\newcommand {\HH} {\mathbb{H}}
\newcommand {\LLe}{\mathbb{L}^{\text{eff}}}

\DeclareMathOperator{\Spec}{Spec}

\DeclareMathOperator{\Freyd}{Freyd}
\DeclareMathOperator{\Cons}{Cons}
\DeclareMathOperator{\Ob}{Ob}
\DeclareMathOperator{\MHM}{MHM}
\DeclareMathOperator{\VMHS}{VMHS}
\DeclareMathOperator{\CMHM}{CMHM}

\newtheorem{thm}{Theorem}[section]
\newtheorem{prop}[thm]{Proposition}
\newtheorem{lemma}[thm]{Lemma}
\newtheorem{cor}[thm]{Corollary}

\newtheorem{rmk}[thm]{Remark}

\numberwithin{equation}{section}
\begin{document}
\title {Motivic sheaves revisited}

\author{Donu Arapura}
\thanks{Author partially supported by a grant from the Simons foundation.}
\address{Department of Mathematics\\
   Purdue University\\
   West Lafayette IN 47907\\
   U.S.A.}

 \maketitle

\begin{abstract}
 The purpose of this paper is to present a simplified construction of the author's category of motivic sheaves 
 \cite{arapura}, and to provide a simplified proof 
 of a theorem of \cite{arapuraL} that the Leray spectral sequence can
 be lifted to this category.
\end{abstract}

Let us recall that given a subfield $k\subset \C$, Nori defined an abelian category of mixed motives $\cM(k)$,
 which received a universal cohomology theory for pairs of $k$-varieties. The book by Huber 
 and M\"uller-Stach  \cite{hm} now gives a fairly  detailed account of this story.
In \cite{arapura}, we generalized Nori's construction to obtain an abelian category $\cM(S)$,
of motivic ``sheaves'' over a $k$-variety $S$. There is a faithful
exact ``Betti'' realization functor $R_B$ from $\cM(S)$ to the category of
constructible sheaves on the analytic space $S_{an}$, and also an
exact functor $R_\ell$ from $\cM(S)$ to the category of constructible
$\ell$-adic sheaves. There is also a Hodge realization from $\cM(S)$ to the heart of a
certain $t$-structure on the derived category of mixed Hodge modules.
This is made explicit when $S$ is a curve in the fourth  section of this paper.
For each projective morphism $f:X\to S$, there exists a motive
$h_S^i(X)\in \cM(S)$ such that $R_B(h^i_S(X))=R^if_*\Z$ and  $R_\ell(h^i_S(X))=R^if_*\Z_\ell$.

The main result of this paper is that there exists a $\delta$-functor
  $h^*:\cM(S)\to \cM(k)$, such that 
$R_B(h^j(M))= H^j(S, R_B(M))$.
Furthermore, if  $f:X\to S$ is projective, there is a spectral sequence
converging to the Nori motive 
$${}_ME_2^{pq} = h^p(h_S^q(X))\Rightarrow h^{p+q}(X)$$
whose  image under $R_B$ is  the Leray spectral sequence.
This is an amalgam of the main theorem of \cite{arapuraL} and  theorem 5.2.1 of \cite{arapura}.
The proof here is simpler than either of the previous proofs. One of the goals of this paper is 
to give more transparent constructions and proofs of some  results from the papers
\cite{arapuraL, arapura}. This is possible, in part, because of some
developments over the intervening years. Nori's Tannakian construction has been
refined by various authors (\cite{blo}, \cite{bp}, 
\cite{hm} and \cite{ivorra}). In particular, using the  set up by Barbieri Viale and
Prest \cite{bp}, we are able to give the more  direct and   natural construction of $\cM(S)$
used here. In addition, some of the complicated homological algebra from
\cite{arapuraL} can be replaced by a technical result due to  de Cataldo and
Migliorini \cite{cm}, which gives a criterion for a map to be an isomorphism in the  filtered derived category.

My thanks to T. Abe and D. Patel for
various comments,  by email or in person, that provided some of the
impetus for writing this follow up.

\section{The N${}^+$ construction}

We will use the term quiver instead of (directed) graph used in
 \cite{arapura}. We will frequently apply category theoretic terminology to quivers.
 In particular, the words ``vertex'' and ``object''
 (respectively ``edge'' and ``morphism'') are used interchangeably.
 The set of objects of $\Delta$ is denoted by $\Ob \Delta$.
 A subquiver $\Delta'\subseteq \Delta$ is full if for any edge  in $\Delta$ joining $e,e'\in \Ob \Delta'$,
is in $\Delta'$.
A morphism, functor, or
 representation between quivers $F:\Delta\to \Delta'$ is a pair of
 functions between vertices and edges which preserves incidence: the
 source/target of $F(e)$ is $F$ applied to the source/target of $e$.
 We say that a diagram of categories is
$2$-commutative, if any two paths between the same vertices are naturally isomorphic.
We will recall the following generalization of Nori's Tannakian construction due
to Barbieri Viale and Prest \cite[pp 207, 214, 215]{bp}, that we will  refer to as the
N${}^+$ construction.

\begin{thm}\label{thm:bp}
  Let $R$ be a commutative ring.
  Given a representation from a quiver to an $R$-linear abelian category
  $F:\Delta\to A$, there exists an $R$-linear abelian category $\A_R(F)$
  and a $2$-commutative diagram
$$
\xymatrix{
 \Delta\ar[r]\ar[rd]^{F} & \A_R(F)\ar[d]^{\phi} \\ 
  & A
}
$$
with $\phi$ $R$-linear faithful and exact. Furthermore, this is universal in the
sense that given any
  other such factorization  $\Delta\to B\to A$, we have a dotted arrow
  as drawn, which is unique up to natural isomorphism, and which  makes the whole diagram 
  $2$-commutative
  $$  
\xymatrix{
 \Delta\ar[r]\ar[d]\ar[rd]^{F} & \A_R(F)\ar[d]\ar@{-->}[ld] \\ 
 {B}\ar[r] & A
}
$$
  
\end{thm}

 It will be useful to briefly summarize the construction, since it will lead to somewhat more refined statements.
 One forms a preadditive
category $R \Delta$ with the same objects as $\Delta$, and for morphisms take 
 the free $R$-module  generated by paths. (In the case where $R$ is not explicitly mentioned, we take $R=\Z$.)
Given an $R$-linear preadditive category $C$, let $[C, R\text{-Mod}]$ denote the
category of $R$-linear additive functors from $C$ to the category of
$R$-modules, and let $[C, \text{R-Mod}]^{fp}$ be  the full subcategory of finitely
presented objects \cite[p 212]{bp}.
Define
$$\Freyd_R(\Delta)=[[R \Delta, \text{R-Mod}]^{fp}, \text{R-Mod} ]^{fp}$$
This is an $R$-linear abelian category. Furthermore, there is a
canonical representation $\Delta\to \Freyd_R(\Delta)$, and $F$ has a
canonical exact extension  $\tilde F:\Freyd_R(\Delta) \to A$.
Then $\A_R(F)$, or $\A(F)$ when $R$ is understood, is the Serre quotient $\Freyd_R(
\Delta)/\ker \tilde F$, where $\ker \tilde F\subset \Freyd_R(
\Delta)$ is the full subcategory with
objects $\{X\mid \tilde F(X)=0\}$.

It should now be clear that the N${}^+$ construction
satisfies the following:

\begin{lemma}
    Let $g:\Delta\to \Delta'$ be a morphism of quivers,
\begin{enumerate}
\item 
There is a
  $2$-commutative diagram
$$
\xymatrix{
 \Delta\ar[r]^{g}\ar[d] & \Delta'\ar[d] \\ 
 \Freyd_R(\Delta)\ar[r]^{G} & \Freyd_R(\Delta')
}
  $$
with $G$ exact. 

\item If  there are representations $F:\Delta\to A$
and $F':\Delta'\to A'$ to $R$-linear abelian categories 
such that $G$ sends objects of $\ker \tilde F$ to $\ker\tilde F'$, then we get an
induced exact functor $\A_R(F)\to \A_R(F')$ such that
$$
\xymatrix{
 \Delta\ar[r]^{g}\ar[d] & \Delta'\ar[d] \\ 
 \A_R(F)\ar[r]& \A_R(F')
}
$$
commutes.
\item If  there are representations $F:\Delta\to A$
and $F':\Delta\to A'$
such that $\ker \tilde F\subseteq \ker \tilde F'$, then $\A_R(F')$ is a Serre quotient of
$\A_R(F)$.
\end{enumerate}
\end{lemma}

\begin{cor}\label{cor:bp}
  Suppose that $F:\Delta\to A$ and $F':\Delta'\to A'$ are two representations to
  abelian categories, that fit into a $2$-commutative diagram
  $$
\xymatrix{
 \Delta\ar[r]^{g}\ar[d] & \Delta'\ar[d] \\ 
 A\ar[r]^{G} & A'
}
  $$
with $G$ exact. Then there is an exact functor $\A(F)\to \A(F')$  fitting into the obvious diagram.

\end{cor}

The following will also be needed later.

\begin{lemma}\label{lemma:dlim}
  Suppose that $\Delta = \bigcup \Delta_i$ is a directed union of
  quivers. If
  $F:\Delta\to A$ is a representation into an $R$-linear abelian
  category, then $\A_R(F)$ is equivalent to the filtered $2$-colimit
  $$\text{2-}\varinjlim_{i} \A_R(F|_{\Delta_i})$$
\end{lemma}

(For the general construction of a filtered $2$-colimit of categories, see  \cite[exp VI
\S 6]{sga4}. In the case, at hand, where $C_i\subseteq C$ is directed
system of subcategories of a fixed category, we can identify
$\text{2-}\varinjlim_{i}  C_i$ with the  union $\bigcup_i
C_i\subseteq C$, i.e. the category whose objects and morphisms
are $\bigcup_i \Ob C_i$ and $\bigcup_i \text{Mor}\, C_i$ respectively.)

\begin{proof}[Sketch]
The family of functors  $\A_R(F|_{\Delta_i})\to \A_R(F)$ induces a  functor
$$\alpha:\text{2-}\varinjlim_{i} \A_R(F|_{\Delta_i})\to \A_R(F)$$
The representations
$$F|_{\Delta_i}: \Delta_i \to \text{2-}\varinjlim_{i} \A_R(F|_{\Delta_i})$$
patch to yield a representation of $\Delta$. Hence, by theorem~\ref{thm:bp},
we get
$$\beta:\A_R(F)\to \text{2-}\varinjlim_{i} \A_R(F|_{\Delta_i})$$
 One checks $\alpha$ and $\beta$ are inverse up to natural equivalence. 
\end{proof}

\section{Effective motivic sheaves}

For the remainder of  the paper, we fix a subfield $k\subset \C$ and a commutative noetherian ring $R$.
By a $k$-variety, we mean  a reduced separated scheme of finite type
over $\Spec k$.  The symbols $S,X,Y$ should be assumed to be $k$-varieties, unless stated otherwise.
Sheaves, and sheaf theoretic operations, should be understood to be with respect to 
the analytic topology $X_{an}=(X\times_{\Spec k}\Spec \C)_{an}$,
unless we explicitly  indicate the  \'etale topology by the decoration $et$.
If $f:X\to S$ is a morphism of $k$-varieties and $Y\subset X$ is a closed
subvariety, then the cohomology of the
pair $(X,Y)$ relative to $S$ with coefficients in a sheaf
$\F$ on   $X$   will be defined by
$$H_S^i(X, Y; \F)=R^if_*j_{X,Y!} \F|_{X-Y}$$
and
$$\mathbb{H}_S(X,Y; \F)= \R f_*j_{X,Y!} \F|_{X-Y},$$
where $j_{X,Y}: X-Y\to X$ is the inclusion.  Note that $H_S^i$ is not cohomology with support in 
$S$. When $S$ is the point $\Spec k$ and $\F$ is constant,
this agrees with what one usually means by cohomology of the pair.
Let us say that a pair $(X\to S, Y)$ has the {\em base change
  property} if  for any morphism $g:S'\to S$ of $k$-varieties, the canonical map gives an isomorphism
$$g^*H_S^i(X, Y;R) \cong H^i_{S'}(X_{S'}, Y_{S'}; R)$$
for all $i$, where $X_{S'} = (X\times_S S')_{red}$ etc.
This property can certainly fail, e.g. for $(\Gm\hookrightarrow
\mathbb{A}^1, \emptyset)$. We give some criteria  for the
property to hold. 

\begin{lemma}\label{lemma:basech}
If $f$ is proper, then  $(f:X\to S,Y)$  has the base change
property. More generally, if
    $(f:X\to S,Y)$ has the base change property and $g:S \to T $ is
  proper, then $(g\circ f:X\to T, Y)$ has the base change property.
\end{lemma}

\begin{proof}
The first statement follows immediately from the proper base change
theorem \cite[thm 2.3.26]{dimca} or \cite[prop 2.6.7]{ks}.
Let us  prove the second.
  Consider the diagram
  $$
  \xymatrix{
 X'\ar[d]^{f'}\ar[r] & X\ar[d]^{f} \\ 
 S'\ar[d]^{g'}\ar[r]^{\pi} & S\ar[d]^{g} \\ 
 T'\ar[r]^{p} & T
}
$$
where both squares are Cartesian. Also let $Y'\subset X'$ denote the
pullback of $Y$. Then by the proper base change theorem  together with
the hypothesis we have
\begin{equation*}
  \begin{split}
    p^*\HH_T(X,Y;R) & = p^*\R g_* \HH_S(X,Y;R)\\
    & = \R g_*'\pi^*\HH_S(X,Y;R)\\
    & = \R g_*' \HH_{S'}(X',Y';R)\\
    & = \HH_{T'}(X',Y';R)
  \end{split}
\end{equation*}
(Equality means that the canonical maps are isomorphisms.)

\end{proof}

Let us say that   $(f:X\to S, Y)$ is {\em controlled}
if there is a factorization  of $f$
$$X\stackrel{f_1}{\to} X'\stackrel{f_2}{\to} S$$
such that
$f_2$ is proper, and $(f_1:X_{an}\to X_{an}', Y_{an})$ is a
topological fibre bundle. To be more explicit, the last condition means that there exists a
topological space $F$, a closed subspace $G\subseteq F$, an open covering $\{U_i\}$ of $X_{an}'$, and
homeomorphisms of pairs
$$(f_1^{-1}U_i, f_1^{-1}U_i\cap Y)\cong (U_i\times F, U_i\times
G)$$
compatible with projection. The notion of being controlled is essentially the same as the
one defined in  \cite[3.2.1]{arapura}.

\begin{lemma}
  If $(f:X\to S, Y)$ is controlled, then it satisfies the base change property.
\end{lemma}

\begin{proof}
  Clearly, if  $(f:X_{an}\to S_{an}, Y_{an})$ is
topological fibre bundle, then it satisfies the base change
property. The general case of the lemma follows from this special case
and lemma~\ref{lemma:basech}.
\end{proof}

\begin{lemma}\label{lemma:prodbasechange}
  Suppose that  $(X_i\to S, Y_i)$  satisfy the base change
  property for $i=1,2$, and suppose that $H_S^*(X_1,Y_1; R)$ are flat
  modules. Then the fibre product
  $$(X_1\times_S X_2\to S, Y_1\times_S X_2\cup X_1 \times_S Y_2)$$
  has the base change property.
\end{lemma}

\begin{proof}
  Let $(X_3, Y_3)$ denote the above fibre product.
    Given a morphism $g:S'\to S$,  the K\"unneth formula gives
  isomorphisms
  \begin{equation*}
    \begin{split}
      g^* H_S^i(X_3,Y_3;R) &= \bigoplus_{j+\ell=i}
      g^*H_S^j(X_1,Y_1;R)\otimes g^*H_S^\ell (X_2,Y_2;R)\\
      &= \bigoplus_{j+\ell=i}
      H_{S'}^j(X_{1S'},Y_{1S'};R)\otimes H_{S'}^\ell(X_{2S'},Y_{2S'};R)\\
      &= H_{S'}^i(X_{3S'},Y_{3S'};R) 
    \end{split}
  \end{equation*}
 A proof of  the  K\"unneth formula can be found in
  \cite[thm 4.3.14]{dimca}; although the reference assumes $R$ is a
  field, it suffices to assume  that the cohomology of one of the factors
  is flat.

\end{proof}

 Let $S$ be a $k$-variety. Define a  quiver $\Delta(S)$ as follows.
When $S$ is connected, the vertices  are triples $(X\to S, Y, i)$
consisting of
\begin{itemize}
\item a  quasi-projective morphism $X\to S$;
\item a closed subvariety $Y\subseteq X$ such that the pair $(X\to
  S,Y)$ has the base change property;
\item a natural number $i\in \N$.
\end{itemize}
One should think of  $(X\to S, Y, i)$ as the symbol representing
$H_S^i(X,Y)$. Let us refer to a pair $(X\to S, Y)$ satisfying the
first two conditions as an {\em admissible} pair.
The set of edges, or morphisms,  of  $\Delta(S)$  is the union of the  following two sets:
\begin{enumerate}
\item[Type I:] Geometric morphisms 
$$(X\to S,Y,i)\to (X'\to S,Y',i)$$ 
for every morphism of $S$-schemes
$X\to X'$ sending $Y$ to $Y'$.

\item[Type II:] Connecting or boundary morphisms 
$$(f:X\to S,Y,i+1)\to (f|_Y:Y\to S,Z,i)$$ 
for every chain $Z\subseteq Y\subseteq
  X$ of closed sets.
\end{enumerate}
When $S$ has several connected components $S_i$, we take $\Delta(S)=
\prod \Delta(S_i)$.

Call a sheaf $\F$ of $R$-modules on   $S_{an}$ {\em $k$-constructible} or simply constructible, if it
has finitely generated stalks and if
there exists a partition $\Sigma=\{Z_i\}$ of $S$ into Zariski locally closed
sets such that $\F|_{Z_{i,an}}$ is locally constant. If $\Sigma$ is given, then $\F$ is
called constructible with respect to $\Sigma$.
The term ``$k$-constructible" is meant
to signify that even though the sheaf is on $S_{an}$, the strata $Z_i$ are defined over $k$.
Let $\Cons(S_{an}, R)$ (or $\Cons(S_{an},\Sigma, R)$) denote the full subcategory of the category of
sheaves of $R$-modules consisting of $k$-constructible sheaves (with respect to $\Sigma$).
It is abelian and $R$-linear. Let $\Delta(S)^{op}$ denote the opposite
quiver, which means that the edges are reversed.
We define a  representation  $H:\Delta(S)^{op}\to \Cons(S_{an}, R)$
which sends $(X\to S,Y,i)$ to
$$H(X\to S,Y,i; R) := H_S^i(X_{an}, Y_{an}; R)$$  
The action of $H$ on edges is as follows. 
We start with the easier case of  a  morphism of type II. 
To 
$$(f:X\to S,Y,i+1)\to (f|_Y:Y\to S,Z,i)$$ 
we assign the connecting map 
$$H_S^i(Y, Z;R)\to H_S^{i+1}(X, Y;R)$$
 induced by the exact sequence
\begin{equation}\label{eq:jXYjXZ}
0\to j_{XY!} R \to j_{XZ!} R \to j_{YZ!} R\to 0 
\end{equation}
For a morphism
$$g:(f:X\to S,Y,i)\to (f':X'\to S,Y',i)$$ 
 of type I, the map on cohomology
 \begin{equation}\label{eq:HtypeI}
H_S^i(X',Y'; R)= R^i f'_*j_{X'Y'!}R  \to R^i f_*j_{XY!}R =H^i_S(X,Y;R)
\end{equation}
is constructed below. We have a  commutative diagram of distinguished triangles
\begin{equation}\label{eq:jXYtriangle}
\xymatrix{
  j_{X'Y'!}R_{X'-Y'}\ar[rd]\ar@{-->}[dd] &  &  R_{Y'}\ar[ll]^{[1]}\ar[dd] \\ 
  &  R_{X'}\ar[ru]\ar[dd] &  \\ 
 \R g_*j_{XY!}R_{X-Y}\ar[rd] &  &  \R g_*  R_{Y}\ar[ll]^{[1]} \\ 
  &  \R g_* R_{X}\ar[ru] & 
}
\end{equation}
The dotted arrow above  induces a map
$$\R f'_*j_{X'Y'!}R \to  \R f'_*\R g_*j_{XY!}R \cong \R f_*j_{XY!}R$$
which gives \eqref{eq:HtypeI}

\begin{rmk}\label{rmk:6fun}
 It should be clear that one can define a representation of $\Delta(S)^{op}$ as above for any theory satisfying 
 Grothendieck's ``six operations" formalism
 (as laid out in \cite[pp 43-44]{bbd} for example).
 In fact, one only needs a theory with direct images and extensions by zero, for which analogues of \eqref{eq:jXYjXZ}
 and \eqref{eq:jXYtriangle} exist.
\end{rmk}

Now we can apply the N${}^+$ construction  to obtain the category
of effective motivic (constructible) sheaves $\M(S, R):=\A_R(H)$.
If $R$ is understood, we write $\M(S) = \M(S,R)$ and  $\M(k) = \M(\Spec k)$. 
The category of motivic sheaves $\cM(S)$ will be built from this in
the next section by inverting a certain object.
If $\Sigma$ is a finite partition of $S$ into locally closed sets, let $\Delta(S,\Sigma)\subset \Delta(S)$
denote the full subquiver of triples $(X\to S, Y,i)$ such that $H_S^i(X,Y)\in \Ob \Cons(S,\Sigma)$.
Then we can consider the subcategory
$$\M(S,\Sigma,R) = \A_R(H|_{\Delta(S,\Sigma)})\subset \M(S,R)$$
of motivic sheaves constructible with respect to $\Sigma$. 
We define the subcategory of motivic local systems as
$$\M_{ls}(S,R) = \M(S,\{S\}, R)$$
The image of $\M_{ls}(S,R)$ under $R_B$ is contained in the category of local systems in the usual sense.

\begin{rmk}\label{rmk:compare}
  Let us compare the story so far with what was done in \cite{arapura}.
\begin{enumerate}
\item In the earlier paper, $\M(S)$ was not considered; $\cM(S)$ was constructed in a single
step. This required a more complicated definition of $\Delta(S)$,
where objects had an extra paramater, and there was an  additional set of  morphisms.

\item Another change in the present definition of $\Delta(S)$ is to  require
  that pairs have the base change property rather than the
  stronger condition that they be controlled. This condition is
   used later for the existence of inverse and direct images (\eqref{eq:pullback} and theorem \ref{thm:leray}).

\item In \cite{arapura}, we only considered the case where $R$ was a field. There
$\cM(S)$ had coefficients in $\Q$.

\item The present construction  corresponds to what were called {\em
    premotivic}  sheaves in \cite{arapura}. There was an additional
  step of forcing $\cM(-)$ to be a stack in the Zariski topology.
  This  could also be done here,
  but  it is not necessary for the present purposes.
  
  \item In \cite{arapura},  the categories $\cM(S,\Sigma)$ were defined first, and $\cM(S)$ was taken to be the $2$-colimit.

\end{enumerate}
\end{rmk}

Let us recapitulate the universal property of the N${}^+$ construction in this context.

\begin{thm}\label{thm:M}
  There is  a faithful exact $R$-linear functor to $R_B:\M(S,R)\to \Cons(S_{an}, R)$, and $H$
  factors through it. This is universal in the sense that given any
  other such factorization  $\Delta(S)\to \mathcal{B}\to \Cons(S_{an},
  R)$ through an $R$-linear abelian category such that the last functor is exact and faithful, we have an essentially unique dotted arrow completing the diagram 
  $$  
\xymatrix{
 \Delta(S)^{op}\ar[r]^{h}\ar[d]\ar[rd]^{H} & \M(S,R)\ar[d]^{R_B}\ar@{-->}[ld] \\ 
 \mathcal{B}\ar[r] & \Cons(S_{an}, R)
}
$$
  
\end{thm}

 We call  the above functor $R_B$, the Betti realization. Given $(X\to S,Y)\in
 \Ob\Delta(S)$, let $h^i_S(X,Y) = h(X\to S, Y, i )$.
 
 Suppose that $f:T\to S$ is a morphism of $k$-varieties. We can
define a morphism of quivers $f^*:\Delta(S)^{op}\to \Delta(T)^{op}$ which takes
$$(X\to S, Y, i) \mapsto (X_T\to T, Y_T, i)$$
Since $(X\to S, Y)$ has the base change property,
$$H(X_T \to T, Y_T, i) \cong f^*H(X\to S, Y, i)$$
Therefore corollary \ref{cor:bp} can be applied to show that there is an
exact functor
\begin{equation}\label{eq:pullback}
f^* :\M(S, R)\to \M(T, R)
\end{equation}
which is compatible with $f^*$ for sheaves under Betti realization
(compare \cite[3.5.2]{arapura}).

 \section{\'Etale realization}

 We want to discuss some other realizations of $\M(S,R) $
 starting with the \'etale realization with finite coefficeints.
 We fix an embedding of the algebraic closure $\bar k\subset
  \C$, and let $\bar X= X\times_{\Spec k} \Spec \bar k$ etc. 
 
 \begin{enumerate}

\item[(R1)] Let $R$ be a finite ring. 
We get a representation  of $\Delta(S)^{op}$ to the category  $ \Cons(S_{et},R)$ of constructible
$R$-modules on $S_{et}$, which sends 
$$(X\to S, Y, i)\mapsto H_S^i(\bar X_{et},\bar Y_{et}, R ):=R^i \bar f^{et}_*\bar j^{et}_{\bar X,\bar Y!} R_{\bar X-\bar Y}$$
This extends to a representation by remark~\ref{rmk:6fun}.
The comparison theorem \cite[ exp XVI, thm 4.1; exp XVII, thm 5.3.3]{sga4} plus theorem \ref{thm:M} implies that there is
an exact faithful  functor $R_{et}:\M(S, R) \to \Cons(S_{et}, R)$  (compare \cite[3.4.6]{arapura}),
called the  \'etale realization.

\end{enumerate}
 
 Before describing the  $\ell$-adic realization,  we need to recall
 some facts about $\ell$-adic sheaves.   A constructible $\ell$-adic
 sheaf on  a scheme $S$ is really an inverse system $\ldots \F_n\to \F_{n-1}\ldots$, where each
 $\F_{n}$ is a constructible sheaf of $\Z/\ell^n\Z$-modules on the
 \'etale topology $S_{et}$, such that each projection yields an isomorphism $\Z/\ell^{n-1}\Z\otimes \F_n\cong \F_{n-1}$.
 The collection of constructible $\ell$-adic sheaves can be made into a $\Z_\ell$-linear
 abelian category $ \Cons(S_{et},\Z_\ell)$ with an appropriate definition  \cite[exp V, VII]{sga5}.
 Ekedahl \cite{ekedahl} has  constructed a triangulated category, that
 we denote by $D_{ek}^b(S,\Z_\ell)$, that behaves
 like the derived category of $ \Cons(S_{et},\Z_\ell)$, and possessess
 Grothendieck's six operations.  More
 precisely, it has a $t$-structure with heart  $ \Cons(S_{et},\Z_\ell)$,
 and a conservative triangulated functor $\Z/\ell\otimes-:D_{ek}^b(S,\Z_{\ell})\to
 D^b(S_{et}, \Z/\ell \Z)$. One has ordinary and extraordinary direct and inverse images in
 $D^b_{ek}(-, \Z_\ell)$ \cite[thm 6.3]{ekedahl}, and these are  compatible with the
 corresponding operations in $D^b(-, \Z/\ell\Z)$. We can therefore define
 $$H_S^i(\bar X_{et},\bar Y_{et}, \Z_\ell )=R^i \bar f^{et}_*\bar j^{et}_{\bar X,\bar Y!} \Z_{\ell,\bar X-\bar Y}$$
 where the above operations and $t$-structure are used  to define $R^i\bar f^{et}_*$ etc.
 
 In this paragaph, let us suppose that $k=\C$. Given  a $\C$-variety $X$,  define the site $X_{cl}$ with objects given by local homeomorphisms
$U\to X_{an}$ and coverings are surjective families $\{U_i\to U\}$. There
 is an obvious map of sites $X_{cl}\to X_{an}$ , which induces an equivalence of  topoi, i.e. of categories of
sheaves \cite[exp XI \S 4]{sga4}. There is a canonical morphism of topoi
$\epsilon:X_{cl}\to X_{et}$ which induces a map from \'etale to classical cohomology.
Ekedahl's construction, which is quite general, can be applied to $X_{cl}$, to yield a triangulated
category $D^b_{ek}(X_{cl}, \Z_\ell)$.  Using the above equivalence of topoi and  \cite[thm 7.2]{ekedahl},
 we obtain equivalences
$$D^b_{ek}(X_{cl}, \Z_\ell)\cong D^b_{ek}(X_{an}, \Z_\ell)\ \cong D^b_c(X_{an}, \Z_\ell)$$
where the category on the right is the  usual constructible derived category. One can see from construction that this
equivalence  is compatible with ordinary and proper direct images. We now come to the key comparison result.

\begin{prop}
 Suppose that $f:X\to Y$ is a morphism of $\C$-varieties, and that
$\F$ is a constructible $\ell$-adic sheaf. Then 
there  are canonical isomorphisms
$$\epsilon^*R^i f^{et}_*\F \cong R^i f^{an}_* \epsilon^*\F$$
$$\epsilon^*R^i f^{et}_!\F \cong R^i f^{an}_! \epsilon^*\F$$
where the direct images on the right are computed in the constructible derived categories.
\end{prop}
 
\begin{proof}
 Consider the distinguished triangle
 $$\epsilon^*R^if^{et}_* \F \xrightarrow{\kappa}R^i f^{an}_* \epsilon^*\F\to C\xrightarrow{[1]}$$
 where $\kappa$ is the canonical map.  The usual comparison
 theorem  \cite[ exp XVI, thm 4.1; exp XVII, thm 5.3.3]{sga4} shows that  $\Z/\ell\otimes C=0$. Since $\Z/\ell\otimes-$ is conservative,
 $C=0$. The proof of the second isomorphism is similar.
\end{proof}

\begin{enumerate}

\item[(R2)] 
We get a representation  of $\Delta(S)^{op}$ to the category  $ \Cons(S_{et},\Z_\ell)$  which sends 
$$(X\to S, Y, i)\mapsto H_S^i(\bar X_{et},\bar Y_{et},\Z_\ell )$$
This extends to a representation by remark~\ref{rmk:6fun}.
The previous proposition plus theorem \ref{thm:M} implies that there is
an exact faithful  functor $R_{\ell}:\M(S,\Z_\ell)\to \Cons(S_{et},\Z_\ell)$  (compare \cite[3.4.6]{arapura}),
called the $\ell$-adic   realization.

\item[(R3)] We can take the product over all primes to get a representation
$$(X\to S, Y, i)\mapsto \prod_\ell H_S^i(\bar X_{et},\bar Y_{et},\Z_\ell )\in \Ob\prod_\ell \Cons(S_{et},\Z_\ell)$$
and this yields a realization
$$\widehat{R}:\M(S,\widehat{\Z})\to \prod_\ell \Cons(S_{et},\Z_\ell)$$

\item[(R4)]   If $R'$ is a  faithfully flat $R$-algebra, there is an $R$-linear exact  change of coefficients functor
 $R'\otimes_R -: \M(S, R)\to \M(S, R')$ fitting into a commutative diagram
   $$
 \xymatrix{
  \M(S, R)\ar[r]^{R'\otimes_R }\ar[d] & \M(S,R')\ar[d] \\ 
 \Cons(S_{an}, R)\ar[r]^{R'\otimes_R } & \Cons(S_{an}, R')
 }
   $$
To see this, 
  define $\M(S,R'/R)$ to be the category whose objects are triples  $(M, L, \phi)$ with $(M,L)\in
  \M(S,R')\times \Cons(S_{an}, R)$ and $\phi: R_B(M)\cong R'\otimes_R L$, and
  with  the  obvious notion of morphisms. Theorem \ref{thm:M} implies the existence of an exact functor $\M(S,R)\to
  \M(S,R'/R)$. Compose this with the projection $\M(S,R'/R)\to \M(S,R')$.
  
\item[(R5)] By combining (R3) and (R4), and taking a projection, one obtains a  realization
  $$\M(S,\Z)\xrightarrow{\widehat{\Z}\otimes} \M(S, \widehat{\Z})\to \Cons(S_{et}, \Z_\ell)$$
  (The same sort of  trick should be applicable to Ivorra's category.)
  
 \end{enumerate}

\section{Hodge realization over a curve}

 Nori constructed a Hodge realization $R_H$ from $\M(k,\Z)$ to the category of
integral mixed Hodge structures using the representation that assigns  to $(X,Y,i)$ the Deligne
mixed Hodge structure on  $H^i(X,Y;\Z)$. Over a general base, things are more complicated.
Saito \cite{saito} has defined his  category of mixed Hodge modules $\MHM(S)$ with the following properties:

\begin{enumerate}
\item  Over a point $\MHM(pt)$ is just
the category of polarizable rational mixed Hodge structures. When $S$ is smooth, objects of $\MHM(S)$ include
polarizable variations of pure Hodge structures, and more generally admissible variations of mixed Hodge structures \cite[thm 0.2]{saito}
\item The category $\MHM(S)$ is abelian and $\Q$-linear. 
There is an exact faithful forgetful functor from $\MHM(S)$ to the category of  rational perverse
sheaves $Perv(S)$ \cite{bbd}.
\item The previous functor extends to a triangulated functor
$D^b\MHM(S)\to D_c^b(S,\Q)$ to the constructible derived
category. 
\item The standard operations $\R f_*, \R f_!,\ldots$ on $D_c^b(-)$  extend to
operations $f^H_*, f_!^H,\ldots$ on $D^b\MHM(-)$ \cite[thm 0.1]{saito}. 

\end{enumerate}

There are two natural $t$-structures \cite[\S 1.3]{bbd}
on $D^b\MHM(S)$. The standard  one has $\MHM(S)$ as
its heart. There is a second $t$-structure on $D^b\MHM(S)$, that we call the classical
$t$-structure,  which corresponds to
the usual one on $D_c^b(S)$ (\cite[appendix C]{arapura}, \cite[rmk 4.6]{saito})). Let 
us call the heart of the classical $t$-structure, the category of
constructible mixed Hodge modules, and denote it by $\CMHM(S)$. 
It possesses a faithful exact functor to $\Cons(S,\Q)$.
To each  of the  $t$-structures, there are associated
cohomological functors ${}^pH^*:D^b\MHM(S)\to \MHM(S)$ and
 ${}^cH^*:D^b\MHM(S)\to \CMHM(S)$ respectively.
In  \cite[3.4.7]{arapura}, we defined
 a Hodge realization functor
$$R_H:\M(S,\Q)\to \CMHM(S)$$
using the representation
$$(f:X\to S,Y,i) \mapsto {}^cH^i f_*^H j^H_{XY!}\Q$$
Note that one can check that this is a representation with the help of  remark~\ref{rmk:6fun}.

The category  $\CMHM(S)$ is constructed abstractly, so the structure
of its objects is not immediately obvious.
We will give a more explicit alternative description of constructible mixed Hodge modules when $S$ is an irreducible
smooth complex curve.  A similar description is possible in general,
but the notation becomes somewhat more cumbersome.
We fix a partition $\Sigma = \{U, p_1,\ldots, p_n\}$ of
$S$ into an open set $U$ and closed points $p_i$. Let $j:U\to S,
i_1:p_1\to S, \ldots$ denote the inclusions.  An admissible variation
of mixed Hodge structures on $U$, consists of a $\Q$-local system
$\F$, plus some other data which imply that all the stalks $\F_x$ are endowed with mixed
Hodge structures. See \cite[\S 3]{sz} for the precise definition.
These form a $\Q$-linear abelian
category $\VMHS(U)$. Given an object $\F\in \Ob \VMHS(U)$,
we can view the perverse sheaf $\F[1]$ as part of a  mixed Hodge module by \cite[thm 0.2]{saito}.
So $\F$ can be viewed as an object of $D^b\MHM(U)$.  We define
$\CMHM(S,\Sigma)$ to be the category with objects
$$\{(\F,M_1,\ldots,\gamma_1,\ldots)\mid \F\in \Ob\VMHS(U), M_a\in \MHM(pt),
\gamma_a:M_a\to H^0(i_a^{H*}j_*^H \F)\}$$
The object $H^0(i_a^{H*}j_*^H \F)$ is a mixed Hodge structure with
underlying vector space  $i_a^*j_*\F$. 
We require that the gluing maps $\gamma_a$
are morphisms of mixed Hodge stuctures. A morphism $(\F,M_1,\ldots)\to (\F',
M_1',\ldots)$
is a collection of morphisms $\F\to \F'$, $M_a\to M_a'$  which are
compatible with the gluing maps. It is not difficult to see that:

\begin{lemma}
 \-
 \begin{enumerate}
\item $\CMHM(S,\Sigma)$ is
a $\Q$-linear abelian category.
\item The functor $F:\CMHM(S,\Sigma)\to \Cons(S,\Sigma)$ which sends
$(\F,M_1,\ldots)$ to
$$\ker \left[j_*\F\oplus\bigoplus_a i_{a*}M_a\to
  \bigoplus_a i_{a*}i_a^*j_*\F\right],$$
where the map is the difference of the adjunction map and $\sum \gamma_a$,
is exact and faithful.
\end{enumerate}
\end{lemma}

Let us outline the construction of the  Hodge realization
$$H_H:\M(S,\Sigma, \Q)\to \CMHM(S,\Sigma)$$
Given $(X,Y,i)\in \Delta(S,\Sigma)$, we need to assign
an object $H_H(X,Y,i)= (\F, M_1,\ldots)$ whose image under $F$ is
$H_S^i(X,Y)$.
Let $X_{p_a} $ and $Y_{p_a}$ denote fibres over $p_a$,
and let $I_a:X_{p_a}\to X$ and $J:X_U\to X$ denote the inclusions.
We set $M_a=H^i(X_{p_a}, Y_{p_a})$ with the Deligne mixed Hodge
structure. Since  $(X,Y,i)\in \Delta(S,\Sigma)$,
$\F=H^i_U(X_U,Y_U)$ is a local system. It carries an admissible variation of mixed Hodge structure, namely
$$({}^pH^i f_*^H j^H_{XY!}\Q)[-1]$$
 This can also be constructed by hand using methods of \cite{sz}, but this is quite a long process.
 The  adjunction maps
$$ I_a^{H*}j_{XY!}^H\Q\to  I_a^{H*}J^H_*J_{H}^* j^{H}_{XY!}\Q$$
induce maps on cohomology
$$H^i(X_{p_a},Y_{p_a})\to i_a^*j_*H^i_U(X_U,Y_U)$$
These are gluing maps. 
For any sheaf $\G$ on $S$, one can check by examining stalks that
$$0\to\G\to j_*j^*\G\oplus\bigoplus_a i_{a*}i_a^*\G\to
  \bigoplus_a i_{a*}i_a^*j_*j^*\G$$
is exact. Applying this observation to $\G= H_S^i(X,Y)$, shows that $F(H_H(X,Y,i))\cong \G$. 
It  remains to check that $H_H$ gives a representation, but this 
 follows with the help of remark~\ref{rmk:6fun}. We note that the
 image of $\M_{ls}(S)$, under the Hodge realization, lies in $\VMHS(S)$.

\section{Motivic sheaves}

Given a category $A$ with an endofunctor $L:A\to A$, following \cite[7.6]{ivorra},
we define a new category $A[L^{-1}]$ with objects $\Ob A\times \Z$, and morphisms
$$Hom_{A[L^{-1}]}((a,n), (b,m))= \varinjlim_i Hom_A(L^{i+n}a,L^{i+m}b)$$
The map $a\mapsto (a,0)$ extends to a functor $A\to A[L^{-1}]$

\begin{lemma}\label{lemma:ivorra}
\-
\begin{enumerate}
\item 
If $A$ is $R$-linear abelian, and $L$ is $R$-linear and exact, then $A[L^{-1}]$ is $R$-linear abelian, and $A\to A[L^{-1}]$
 is exact.  

\item If $L$ is an equivalence, then $A$ is equivalent to $A[L^{-1}]$.
\item There exists a $2$-commutative diagram
 $$
 \xymatrix{
  A\ar[r]^{L}\ar[d] & A\ar[d] \\ 
 A[L^{-1}]\ar[r]^{L'} & A[L^{-1}]
}
 $$
 where  $L'$  is an equivalence. 
\item Given a functor $F:A\to B$ and a $2$-commutative diagram
  $$
 \xymatrix{
  A\ar[r]^{L}\ar[d] & A\ar[d] \\ 
 B\ar[r]^{L''} & B
}
 $$
there exists a $2$-commutative diagram
$$
 \xymatrix{
  A\ar[r]\ar[d] & A[L^{-1}]\ar[d] \\ 
 B\ar[r] & B[(L'')^{-1}]
}
 $$
\end{enumerate}
\end{lemma}

\begin{proof}
 The first statement is \cite[lemma 7.4]{ivorra}. If $L$ is an
 equivalence with quasi-inverse $L^{-1}$, one sees that $(a,n)\cong
 (L^na,0)$, and that
$$Hom_{A[L^{-1}]}((a,0),(b,0))\cong Hom_A(a, b)$$
So the functor $A\to A[L^{-1}]$ is essentially surjective and fully faithful.
One defines $L'(a,n)= (La,n+1)$, and checks this gives an auto-equivalence
of $A[L^{-1}]$ extending $L$.
 The last  statement is clear from the construction.
\end{proof}

\begin{cor}
If the functor $L''$  in (4) is an equivalence, there is a $2$-commutative diagram
 $$
 \xymatrix{
  A\ar[r]\ar[rd] & A[L^{-1}]\ar[d]\\ 
  & B
}
 $$
\end{cor}

Call $(X\to S, Y, i)\in \Delta(S)$ 
{\em cellular} if $X/S$ is affine with equidimensional fibres,  such that
$$H^j_S(X,Y;R)=0$$
unless $j=i$, and $H^i_S(X,Y;R)$ is flat over $R$.
A basic example  of cellular object is $({\Gm}_S\to S, \{1\},1)$.
We refer to  the corresponding motive $h_S^1(\mathbb{G}_{m,S}, 1)$ as the Lefschetz motive.
In the absolute case, a cellular object is  what Nori calls a ``good
pair"  \cite{hm}.

\begin{lemma}
Given a cellular object $(Z\to S, W, j)$, let $M=H^{j}_S(Z,W;R)$.
The map $\zeta:\Ob \Delta(S)\to \Ob\Delta(S)$ given by
\begin{equation}
  \label{eq:prodcellular}
(X\to S, Y,i)\mapsto  (Z\times_S X\to S, W\times_S X\cup Z \times_S Y, j+i)  
\end{equation}
is a morphism of quivers.
The K\"unneth isomorphism
$$H_S^{i+j}(Z\times_S X, W\times_S X\cup Z \times_S Y; R) \cong M\otimes_R H^i(X, Y; R)$$
 renders the diagram
\begin{equation}\label{eq:lambda}
\xymatrix{
 \Delta(S)^{op}\ar[r]^{\zeta^{op}}\ar[d] & \Delta(S)^{op} \ar[d] \\ 
 \Cons(S_{an}, R)\ar[r]^{M\otimes} & \Cons(S_{an}, R)
}
\end{equation}
$2$-commutative. 
\end{lemma}

We omit the proof, but remark that the right side of
\eqref{eq:prodcellular} lies in $\Ob\Delta(S)$ by lemma~\ref{lemma:prodbasechange}.
It follows from this lemma and corollary \ref{cor:bp}, that if $(Z\to
S, W, j)$ is cellular, then we 
 can construct an induced exact endofunctor
 $$ h_S^j(Z,W)\otimes -:\M(S,R)\to \M(S, R)$$
Define the exact endofunctor
$\LLe :\M(S,R)\to \M(S,R)$ by $\LLe=h^1_S(\mathbb{G}_{m,S}, 1)\otimes -$.
Set
 $\cM(S,R) := \M(S,R)[(\LLe)^{-1}]$. Then there exists a $2$-commutative diagram
 $$
 \xymatrix{
 \M(S,R)\ar[r]^{\LLe}\ar[d] & \M(S,R)\ar[d] \\ 
 \cM(S,R)\ar[r]^{\LL} & \cM(S, R)
}
 $$
 with  $\LL$  invertible. Furthermore, $\cM(S,R)$ is the
 universal such category. We refer to this as the category of {\em motivic
 sheaves with coefficients in $R$}, and $\cM(S)= \cM(S,\Q)$ simply as the category of 
   motivic sheaves. The category 
$\M(S)$ is good enough for most purposes, but inverting $\LL$ becomes
important in  situations where one considers duals.

The $2$-commutativity of \eqref{eq:lambda} shows that
 there is a natural isomorphism $R_B\circ \LLe\cong R_B$. Therefore $R_B$
 extends to an exact functor $\cM(S,R)\to \Cons(S_{an}, R)$ by the
 universal property. The construction of
 $\M(S,R)[(\LLe)^{-1}]$  shows that this is faithful. 
 We note that the category $\CMHM(S)$ is stable under the operations $M\mapsto M\otimes \Q(\pm 1)$.

\begin{lemma}\label{lemma:real}
 The functors $R_\ell, R_H$ and $f^*$  extend to $\cM(-)$. Tate twists in $\cM(S)$ are compatible
 with these operations in the sense that
 $$R_\ell \circ \LL \cong \Z_\ell(-1)\otimes R_\ell$$
 $$ R_H\circ \LL \cong \Q(-1)\otimes R_H$$
  $$ f^* \circ \LL \cong \LL\circ f^*$$
\end{lemma}

\begin{proof} By K\"unneth, one gets an isomorphism of \'etale cohomology
$$H_S^{i+j}(\Gm\times_{\bar S} \bar X, 1\times_{\bar S} \bar X\cup \Gm
\times_{\bar S} Y;
\Z_\ell) \cong \Z_\ell(-1)\otimes_{\Z_\ell} H^i(\bar X, \bar Y; \Z_\ell)$$
Therefore
$$R_\ell \circ \LLe \cong \Z_\ell(-1)\otimes R_\ell$$
Similarly one checks that
 $$ R_H\circ \LLe \cong \Q(-1)\otimes R_H$$
  $$ f^* \circ \LLe \cong \LLe\circ f^*$$
 on $\M(-)$. The lemma is a formal consequence of these identities.
\end{proof}
 
The category $\cM(k)=\cM(\Spec k)$ is precisely Nori's category of mixed motives.
Let $\cM^{pure}(k)\subset \cM(k)$ denote the full subcategory generated by subquotients of $h^i(X)$, with $X$ smooth and projective.
This subcategory can be related to other  constructions.
 Andr\'e defined a category of pure motives \cite{andre}, by replacing algebraic cycles in Grothendieck's construction
 by his motivated cycles.
 If one  assumes the standard conjectures, then  Andr\'e's and Grothendieck's categories of motives would coincide.

\begin{thm}\label{thm:andre}
 The category $\cM^{pure}(k,\Q)$ is equivalent to Andr\'e's category of pure motives.
\end{thm}

\begin{proof}
 See \cite[thm 6.4.1]{arapura} or \cite[thm 10.2.7]{hm}.
\end{proof}

\begin{cor}\label{cor:andre}
 The category $\cM^{pure}(k, \Q)$ is semisimple abelian.
\end{cor}

\begin{proof}
 \cite[thm 0.4]{andre}.
\end{proof}

\section{Background on cellular decomposions}

We recall some background results from
\cite{arapuraL,arapura}, along with some simplifications. First, we recall  Jounalou's trick
\cite[lemma 1.5]{jounalou}.

\begin{lemma}[Jouanalou]\label{lemma:jou}
If $S$ is a quasi-projective variety then there exists an affine
variety $T$ and a smooth morphism $\pi:T\to S$ which is Zariski
locally isomorphic to $\mathbb{A}^n_S$ for some $n$.
\end{lemma}

Fix $\pi:T\to S$ as above. Since $\C^n$ is contractible,
then given $\F\in \Cons(S_{an})$, we have an isomorphism
$H^i(T, \pi^*\F)\cong H^i(S_{an}, \F)$ for each $i$. So for our purposes, we lose
nothing by working on $T$. 
The next result is key, and is a consequence of Beilinson's basic
lemma.  First, we need some notation. Given a sheaf $\F$ on $T$, and
closed sets $T_1\subset T_2\subset T$,
let
$$H^i(T_2, T_1; \F) = H^i(T_2, j_{T_2T_1!} (\F|_{T_2-T_1}))$$
where
$$j_{T_2T_1}:T_2-T_1\hookrightarrow T_2$$
is the inclusion.
When $\F$ is constant, this is just the cohomology of the
pair in the usual sense.

\begin{lemma}\label{lemma:beil}
Given $\F\in \Cons(T_{an})$,
 there exist a chain
   $$\emptyset = T_{-1}\subset T_0\subset T_1\subset \dots =T$$
  of equidimentional closed sets with $\dim T_i= i$,
   such that for all $a$
   $$H^i(T_a, T_{a-1}; \F)=0$$
   unless $i=a$.  Furthermore, $T_\dt$ can be chosen to refine a given chain. 
\end{lemma}

\begin{proof}
  \cite[lemma 3.7]{arapuraL}.
\end{proof}

Let us say that an admissible  pair $(X\to S,Y)$ is {\em cellular }
with respect the chain $T_\dt$ if for all $a$
   \begin{equation}
     \label{eq:cellular}
        H^i(T_a, T_{a-1}; \pi^*H_S^*(X,Y))=0\text{ if } i\not=a 
   \end{equation}
    We say that  an object $(X\to S, Y,j)$ is cellular with
    respect to  $T_\dt$,
    if $(X\to S, Y)$ is.
 Let $\Delta(S,T_\dt)\subset \Delta(S)$ be
   the full sub quiver of triples $(X\to S,Y, j)$ cellular with respect to $T_\dt$.

   \begin{cor}\label{cor:adapt}
     $\Delta(S)$ is a directed union of
 $\Delta(S, T_\dt)$, as $T_\dt$ runs over various
 chains (with $T$ fixed as above).
   \end{cor}

Fix a commutative noetherian ring $R$. Cohomology should be understood to take
values in $R$ if coefficients are not specified.
Suppose that  $(X\to S, Y)$ is an admissible pair which is cellular  with
respect to   $T_\dt$. 
Let  $\F= \pi^*\HH_S (X,Y;R)\in D^b(T_{an}, R)$, and
$ \G =\G_{(X,Y)}=\R \Gamma\F\in D^b(R\text{-mod})$. We construct filtrations  on $\G$ by $P^\dt(\G)=\R \Gamma\tau_{\le-\dt}\F$,
and  
$$F^\dt \G= \R \Gamma j_{T T_\dt !}j_{T T_\dt}^*\F$$
Define a new filtration $Dec(F)$ by  {\em d\'ecalage} \cite[1.3.3]{deligne-hodge}.
 
\begin{lemma}\label{lemma:leray}
The identity of $\G$ induces an isomorphism in the filtered derived category
  $$(\G, P)\cong (\G, Dec(F))$$
\end{lemma}

\begin{proof}
This follows from  de Cataldo-Migliorini \cite[prop 5.6.1]{cm}. The conditions of their
proposition hold because of \eqref{eq:cellular}.
 
\end{proof}

As a corollary, we get an isomorphism of spectral sequences associated
to the two filtrations
\begin{equation*}
  \begin{split}
    {}_P E_1^{pq} &= H^{p+q}(Gr^p_P \G)\\ 
\cong {}_{Dec(F)}E_1^{pq } &=
H^{p+q}(Gr^p_{Dec(F)} \G)
  \end{split}
\end{equation*}
After reindexing  ${}_P
E_1^{pq} = {}_LE_2^{2p+q,-p}$, the first
spectral sequence  can be identified with the Leray spectral sequence
$${}_LE_2^{pq}=H^p(S, H_S^q(X,Y;R) )\Rightarrow H^{p+q}(X,Y;R)$$
c.f. \cite[1.4.8]{deligne-hodge}.
A similar change of variable  leads
to an isomorphism $ {}_{Dec(F)}E_1^{pq}={}_FE_{2}^{2p+q,-p}$ \cite[1.3.3.2]{deligne-hodge}. The
first page  of the last spectral sequence is
 $${}_FE_1^{pq} = H^{p+q}(Gr_F^p\G)\cong H^{p+q}(  X_{  T_p},  Y_{  T_p}\cup   X_{  T_{p-1}}) $$
where the differentials are connecting maps. The whole spectral sequence can be constructed, in the usual manner \cite[\S 5.9]{weibel},
 from  the exact couple 
 \begin{equation}
   \label{eq:excouple}
   \begin{split}
     E_1&= \bigoplus_p H^*(X_{  T_p},  Y_{  T_p}\cup   X_{  T_{p-1}})\\
D_1&= \bigoplus_p H^*(X,  Y\cup   X_{  T_{p-1}}) 
   \end{split}
 \end{equation}
with maps
$$
\xymatrix{ D_1 \ar[rr] & & D_1\ar[ld]\\  &E_1\ar[lu] &}
$$
coming from the long exact sequence associated to the triples  $(X,Y\cup X_{T_p}, Y\cup X_{T_{p-1}})$.
For each $i$, let us write $\calK(i)^\dt=\calK(X,Y,i)^\dt = E_1^{i\dt}$, i.e.
\begin{equation}\label{eq:Ki}
\calK(i)^\dt= H^{i}( X_{ T_0}, Y_{ T_0}\cup X_{ T_{0-1}})
 \to H^{i+1}(  X_{  T_1},  Y_{  T_1}\cup   X_{  T_{1-1}})\to \ldots  
\end{equation}

To summarize:

\begin{cor}\label{cor:leray}
With the same assumptions as above,
  there is an isomorphism of spectral sequences ${}_LE_2^{pq}\cong
  {}_FE_{2}^{pq}$.
In particular, there is an isomorphism of $R$-modules
$$
  \phi: \calH^j(\calK(i)^\dt)\cong H^j(S, H_S^i(X,Y))
 $$
where $\calH^j$ stands for the $j$th cohomology module of a complex.
\end{cor}
 
We need to understand the naturality  properties of above the isomorphism.
Given a morphism of $(X'\to S,Y')\to
(X\to S,Y)$ of pairs, we have a morphism $\G_{(X,Y)}\to \G_{(X', Y')}$
compatible with the filtrations $P,F, Dec(F)$ and the isomorphism  of lemma
\ref{lemma:leray} (which is just the identity!).
In particular, we can
conclude that we have a morphism
$$\calK(X,Y,i)^\dt\to \calK(X',Y',i)^\dt$$
and a commutative diagram
\begin{equation}\label{eq:phi1}
\xymatrix{
 \calH^j(\calK(X,Y,i)^\dt)\ar[d]\ar[r]^{\phi} & H^j(S, H_S^i(X,Y))\ar[d] \\ 
 \calH^j(\calK(X',Y',i)^\dt)\ar[r]^{\phi} & H^j(S, H_S^i(X',Y'))
}
\end{equation}

Let $Z\subseteq Y\subseteq X$ be a chain of closed sets. Then 
the exact sequence
$$0\to j_{XY!} R \to j_{XZ!} R \to j_{YZ!} R\to 0$$
gives rise to a distinguished triangle
$$\G_{(X,Y)}\to \G_{(X,Z)}\to \G_{(Y,Z)}\to \G_{(X,Y)}[1]$$
The last morphism is compatible with the filtrations leading to 
a commutative diagram
\begin{equation}\label{eq:phi2}
\xymatrix{
 \calH^j(\calK(Y,Z,i)^\dt)\ar[d]\ar[r]^{\phi} & H^j(S, H_S^i(Y,Z))\ar[d] \\ 
 \calH^j(\calK(X,Y,i+1)^\dt)\ar[r]^{\phi} & H^j(S, H_S^{i+1}(X,Y))
}
\end{equation}
 
  \section{Motivic Leray}

Fix a subfield $k\subset \C$, and commutative noetherian ring $R$. Let
$\cM(S)=\cM(S,R)$ for the rest of this section. Here is the main result of the paper. It refines
theorem 3.1 of \cite{arapuraL}, although the strategy of proof is
closer to that of  \cite[thm 5.2.1]{arapura}.

 \begin{thm}\label{thm:leray}
   Let $S$ be a  quasiprojective $k$-variety. Then
   there exists a $\delta$-functor
   $\{h^j :\cM(S)\to \cM(k)\}_{j=0,1\ldots}$, such that for each $j$, the diagram
   $$
\xymatrix{
 \cM(S)\ar[r]^{h^j}\ar[d]^{R_B} & \cM(k)\ar[d]^{R_B} \\ 
 \Cons(S_{an})\ar[r]^{H^j} & R\text{-mod}
}
   $$
   $2$-commutes. Given a controlled pair $(f:X\to S, Y)$, there exists
a spectral sequence
$${}_ME_2^{pq} = h^p(h_S^q(X,Y))\Rightarrow h^{p+q}(X,Y)$$
in $\cM(k)$
whose image under $R_B$ is isomorphic to the Leray spectral sequence
$${}_L E_2^{pq} = H^p(S, H_S^q(X,Y))\Rightarrow H^{p+q}(X,Y)$$

 \end{thm}

We will defer the proof until we have established some preliminary results.
Define  a new category $C$ whose objects are triples
$$(K^\dt, \F, \phi: H^*(S,\F)\cong R_B\circ \cH^*(K^\dt))$$
where $K^\dt\in \Ob C^b(\M(k))$, $\F\in \Cons(S,R)$, and
$\phi$ an isomorphism of graded $R$-modules. A morphism
$(K_1^\dt, \F_1,\phi_1)\to (K_2^\dt,\F_2, \phi_2)$ is  pair of morphisms $K_1^\dt\to K_2^\dt$, $\F_1\to \F_2$
which are compatible under $\phi_i$.

\begin{lemma}
 The category $C$ is   abelian, and the projections $p_1:C\to C^b(\M(k))$ and $p_2:C\to \Cons(S)$ are exact.
 The induced functor from $\overline{C}= C/\ker p_2$ to $\Cons(S)$ is exact and faithful.
 The functor $\cH^i\circ p_1:C\to \M(k)$ factors through the quotient $\overline{C}$.

\end{lemma}

\begin{proof}
 The first two statements are  straightforward and completely formal, so let us focus on the last. Let $\Sigma$ be the set of morphisms of $C$ whose
  kernel and cokernel lie in $\ker p_2$. Then by construction $\overline{C}$ is the localization $\Sigma^{-1}C$. So it suffices
 to prove that $\cH^i\circ p_1$ takes $\Sigma$ to the set of isomorphisms.
 Let $f:(K_1^\dt, \F_1,\phi_1)\to (K_2^\dt,\F_2, \phi_2)$ be in $\Sigma$. Then $f$ induces an isomorphism $\F_1\cong \F_2$,
 and therefore an isomorphism $H^i(\F_1)\cong H^i(\F_2)$. It follows that $f$ induces an isomorphism $\cH^i(K_1)\cong \cH^i(K_2)$.
\end{proof}

\begin{prop}
There is a  $\delta$ functor $h^*:\M(S)\to \M(\Spec k)$ such that
 $$
\xymatrix{
 \M(S)\ar[r]^{h^j}\ar[d]^{R_B} & \M(k)\ar[d]^{R_B} \\ 
 \Cons(S_{an})\ar[r]^{H^j} & R\text{-mod}
}
   $$
   $2$-commutes. 
\end{prop}

 \begin{proof}
    By lemma~\ref{lemma:jou}, we can find an affine
   variety $T$ and an affine space bundle $p:T\to S$. We fix this
   choice. By corollary \ref{cor:adapt},  $\Delta(S)$ is a directed union of
 $\Delta(S, T_\dt)$. 
 Therefore by   lemma~\ref{lemma:dlim},  $\M(S)$ is the filtered $2$-colimit of
  the categories
 $$\M(S, T_\dt) = \A(H|_{\Delta(S,T_\dt)})$$
Thus it suffices to define $h^j$ on these categories, and verify
compatibility under refinement.

Given an object $(X,Y, i)\in \Ob\Delta(S, T_\dt)$, let
$K_{T_\dt}(X,Y,i)$ denote the sequence of motives in $\M(k)$ given by
$$ h^{i}( X_{ T_0}, Y_{ T_0}\cup X_{ T_{0-1}})
 \stackrel{d}{\to} h^{i+1}(  X_{  T_1},  Y_{  T_1}\cup   X_{
   T_{1-1}}) \stackrel{d}{\to}  \ldots 
 $$
where the maps $d$ are connecting maps.  One can check immediately that
$R_B(d^2)=0$, so $d^2=0$ because $R_B$ is faithful. Therefore $K_{T_\dt}(X,Y,i)$ is an object in the abelian
category of bounded chain complexes $C^b(\M(k))$. Its image
$R_B(K(X,Y,i))\in C^b(R\text{-mod})$ is the complex $\calK_i^\dt$ constructed in
\eqref{eq:Ki}. We define 
$$F_{T_\dt}(X,Y,i)= (K_{T^\dt}(X,Y,i), H^i_S(X,Y), \phi )\in \Ob C$$
where $\phi$ comes from corollary \ref{cor:leray}.

We claim that $F_{T_\dt}:\Delta(S,T_\dt)^{op}\to C$ is a representation.
Given a morphism $(X', Y', i)\to (X,Y,i)$ of type I, one gets a diagram
$$
\xymatrix{
  h^{i}( X_{ T_0}, Y_{ T_0}\cup X_{ T_{0-1}})
\ar[r]\ar[d] & h^{i+1}(  X_{  T_1},  Y_{  T_1}\cup   X_{
   T_{1-1}})\ar[r]\ar[d] &  \\ 
  h^{i}( X_{ T_0}', Y_{ T_0}'\cup X_{ T_{0-1}}')\ar[r] & h^{i+1}(  X_{  T_1}',  Y_{  T_1}'\cup   X_{
   T_{1-1}}')\ar[r] & 
}
$$
It commutes because it does so after applying $R_B$. Therefore we have 
morphism $K_{T_\dt}(X,Y,i)\to K_{T_\dt}(X', Y',i)$. This can be
completed to a morphism  $F_{T_\dt}(X,Y,i)\to F_{T_\dt}(X', Y',i)$ by
\eqref{eq:phi1}.

Similarly, given a morphism of type II associated to a triple
$Z\subseteq Y\subseteq X$, one gets a commutative diagram
$$
\xymatrix{
  h^{i}( X_{ T_0}, Y_{ T_0}\cup X_{ T_{0-1}})
\ar[r]\ar[d] & h^{i+1}(  X_{  T_1},  Y_{  T_1}\cup   X_{
   T_{1-1}})\ar[r]\ar[d] &  \\ 
  h^{i+1}( Y_{ T_0}, Z_{ T_0}\cup Y_{ T_{0-1}})\ar[r] & h^{i+2}(  Y_{  T_1},  Z_{  T_1}\cup   Y_{
   T_{1-1}})\ar[r] & 
}
$$
This can be extended to  a morphism
$$(K_{T_\dt}(X,Y,i), H_S^i(X,Y), \phi) \to (K_{T_\dt}(Y,Z,i+1), H_S^{i+1}(Y,Z), \phi) $$
using \eqref{eq:phi2}.
Thus $F_{T_\dt}$ is a representation as claimed. 

Let $\bar F_{T_\dt}:\Delta(S,T_\dt)^{op}\to \overline{C}$ be the compostion of $ F_{T_\dt}$ with the quotient map.
Since $H$ factors through $\bar F_{T_\dt}$, by theorem \ref{thm:M}, it  extends
to an exact functor
$$\bar F_{T_\dt}:\M(S,T_\dt)\to \overline{C}$$
Let $h_{T_\dt}^j$ denote the composite
\begin{equation}
  \label{eq:KS}
\M(S,T_\dt)\stackrel{\bar F_{T_\dt}}{\longrightarrow} \overline{C} \stackrel{\cH^j\circ p_1}{\longrightarrow}\M(k)
\end{equation}
where the second arrow comes from the previous lemma.
 This forms a $\delta$-functor, since $\cH^j\circ p_1$ does.

As noted already, by corollary~\ref{cor:leray}
\begin{equation}\label{eq:HjK}
  R_B(h_{T_\dt}^j(h^i_S(X,Y)))\cong \cH^j(R_B(K_{T_\dt}(X,Y,i)) \cong H^j(S, H^i_S(X,Y))
\end{equation}
If $T_\dt'\subseteq T_\dt$, then one has a map of quivers $\Delta(S,T_\dt)\to \Delta(S,T_\dt')$.
and a corresponding  map of complexes 
\begin{equation}\label{eq:KS}
K_{T_\dt}(X,Y,i)\to K_{T_\dt'}(X,Y,i)
\end{equation}
This is a quasi-isomorphism by \eqref{eq:HjK}.
Therefore $h_{T_\dt}^j$ is compatible with refinement, so it
extends to a functor $h^j$ on the $2$-colimit $\M(S)$.

 \end{proof}

\begin{proof}[Proof of theorem \ref{thm:leray}]
We first  prove that the functor $h^j$ constructed in the last
proposition has an extension  $\cM(S)\to \cM(k)$ with the same properties.
Given a complex $K^\dt\in C^b(\M(k))$, observe that
$$R_B(\LLe K^\dt)\cong H^1(\Gm,1)\otimes_R R_B(K^\dt)\cong R_B(K^\dt)$$
Define $\LLe: C\to C$ by
$$\LLe( K^\dt, \F, \phi: H^*(S,\F)\cong R_B\circ \cH^*(K^\dt)) )= (\LLe K, \F, H^*(S,\F)\cong R_B\circ \cH^*(\LLe K^\dt)) )$$
The composite $\LLe:C\to \overline{C}$ factors through $\overline{C}$.

Let $\lambda:\Ob \Delta(S)\to \Ob\Delta(S)$ be given by
$$(X\to S, Y,i)\mapsto  (\Gm\times X\to S, \Gm\times Y\cup 1\times X, i+1)$$
Then
$$K(\lambda(X,Y,i))\cong \LLe K(X,Y,i)$$
One can check that the diagram
$$
\xymatrix{
 \Delta(S,T_\dt)^{op}\ar[r]^{\bar F}\ar[d]\ar[rd]^{\lambda^{op}} & \overline{C}\ar[rd]^{\LLe} &  \\ 
 \M(S,T_\dt)\ar[rd]^{\LLe}\ar[ru] & \Delta(S,T_\dt)^{op}\ar[r]^{\bar F}\ar[d] & \overline{C} \\ 
  & \M(S,T_\dt)\ar[ru] & 
}
$$
$2$-commutes. This implies that  $\bar F:\M(S,T_\dt)\to \overline{C}$ extends to a functor
$\cM(S,T_\dt)\to \overline{C}$.  Composing with $\calH^j\circ p_1$, and passing to the colimit, 
gives  an extension $h^j:\cM(S)\to \cM(k)$ such that $\LL\circ
h^j\cong h^j\circ \LL$.
Any object of $M\in \Ob \cM(S)$ is isomorphic to $\LL^n M'$ with $M'\in \Ob
\M(S)$. Therefore
$$R_B(h^j( M))\cong R_B(\LL^n h^j(M')) \cong  R_B(h^j(M'))\cong H^j(R_B(M))$$

Define an exact
couple in $\cM(k)$ by
$${}_ME_1= \bigoplus h^*(X_{  T_p},  Y_{  T_p}\cup   X_{  T_{p-1}}) $$
$${}_MD_1= \bigoplus h^*(X,  Y\cup   X_{  T_{p-1}}) $$
with maps
$$
\xymatrix{ D_1 \ar[rr] & & D_1\ar[ld]\\  &E_1\ar[lu] &}
$$
induced by connecting maps as in \eqref{eq:excouple}. This generates a
spectral sequence ${}_ME_1^{pq}$.
Then the image
of this exact couple under $R_B$ is \eqref{eq:excouple}.  Therefore
$R_B({}_ME_1^{pq}) \cong {}_FE_1^{pq}$. Therefore by corollary $R_B({}_ME_2^{pq}) \cong {}_LE_2^{pq}$.

\end{proof}

The proof actually gives a bit more than what was stated.

\begin{cor}[of proof]
  With the same assumptions as in the theorem,
 there is a well defined triangulated
functor $r\Gamma:D^b\cM(S)\to D^b \cM(k)$, such  $h^j M = \cH^j(
r\Gamma M)$ for any $M\in \cM(S)$.
\end{cor}

\begin{proof}
The functor $p_2\circ \bar F_{T_\dt}$ extends to
  an exact functor
  $$C^b\M(S,T_\dt)\to C^b(C^b(\M(k)))$$
from the category of single complexes to double complexes.  Composing
with the  total
complex, and projection, yields a functor
$$C^b\M(S,T_\dt)\to D^b\M(k)$$ 
The map \eqref{eq:KS} is a quasi-isomorphism by
\eqref{eq:HjK}. Therefore the above map passes to the  $2$-colimit
$$C^b\M(S)\to D^b\M(k)$$ 
This factors through $D^b\M(S)$, and
satisfies $h^jM = \cH^j(
r\Gamma M)$. One can check that this commutes with $\LLe$, therefore extends to $r\Gamma:D^b\cM(S)\to D^b \cM(k)$.
\end{proof}

\begin{thm}
Let $X$ be a
  smooth and projective variety. Let $f:X\to
  S$ be a surjective projective morphism to another variety $S$, and
  assume   that either $f$ is smooth and projective, or that $S$ is a smooth
  projective curve. Then there is a noncanonical decomposition
$$h^i(X) \cong \bigoplus_{p+q=i} h^q(h^p_S(X))$$ 
in $\cM(k,\Q)$.
\end{thm}

\begin{proof}
   The Leray spectral sequence
$${}_LE_2 = H^p(S, H^q_S(X;\Q)) \Rightarrow H^{p+q}(X;\Q)$$
degenerates at $E_2$, either by Deligne \cite[thm 1.5]{deligneL} when $f$ is
smooth and projective, or by Zucker \cite[cor 15.15]{zucker} when $S$ is a curve.
This means that the differentials $d_2, d_3,\ldots$ are all
zero. Therefore the same holds for the  spectral sequence ${}_ME_2$, constructed in the
previous theorem. It follows that there is a filtration $L^p$ on
$h^i(X)$ such that
$$Gr_L^p h^i(X)=  h^q(h^p_S(X))$$ 
Since $h^i(X)$ lies in $\cM^{pure}(k)$, and this category is
semisimple  by corollary \ref{cor:andre},
 it follows that the maps
$$h^i(X)\leftarrow L^p h^i(X)\to Gr^p_L h^i(X)$$
split.  So we obtain an isomorphism
$$h^i(X) \cong \bigoplus_{p+q=i} h^q(h^p_S(X))$$ 
\end{proof}



\begin{thebibliography}{123}

 \bibitem[A]{andre} Y. Andr\'e, {\em Pour une th\'{e}orie inconditionnelle
     des motifs}, {Inst. Hautes \'{E}tudes Sci. Publ. Math.}, 83 (1996)
   
 \bibitem[A1]{arapuraL} D. Arapura, {\em The Leray spectral sequence is
   motivic,} Invent. Math. 160 (2005), no. 3, 567-589
   
 \bibitem[A2]{arapura} D. Arapura, {\em  An abelian category of motivic
   sheaves,} Adv. Math. 233 (2013), 135-195. 

\bibitem[SGA4]{sga4}  M. Artin, A. Grothendieck, J-L Verdier, {\em Theorie de topos et
 cohomologie \'etale de sch\'emas} Springer LNM 269, 270, 305 (1972-1973)

   
 \bibitem[BLO]{blo} L. Barbieri Viale,  L. Lafforgue, O. Carmelo,
  {\em  Syntactic categories of Nori motives,} Selecta Math (2018), 3619-3648.
   
   
 \bibitem[BP]{bp} L. Barbieri Viale, M. Prest, {\em Definable categories
   and $\mathbb{T}$-motives}, Rend. Semin. Mat. Univ. Padova 139 (2018), 205-224
   

\bibitem[BBD]{bbd} A. Beilinson, J. Bernstein, P.  Deligne, {\em
    Faisceaux pervers}, Ast\'{e}risque, 100, Soc. Math. France, (1982)
 
   
   
\bibitem[CM]{cm} M. de Cataldo, L. Migliorini, {\em The perverse filtration and 
 the Lefschetz hyperplane theorem, } Annals of Math 171 (2010), 2089-2113. 

   
 
\bibitem[D1]{deligneL} P. Deligne, {\em Th\'{e}or\`eme de {L}efschetz et crit\`eres de d\'{e}g\'{e}n\'{e}rescence de
              suites spectrales}, Inst. Hautes \'{E}tudes
            Sci. Publ. Math. 35 (1968)

 \bibitem[D2]{deligne-hodge} P. Deligne, {\em Theorie de Hodge II},
  Inst. Hautes \'Etudes Sci. Publ. Math 40, (1971)

 
 
 \bibitem[Di]{dimca} A. Dimca, {\em Sheaves in topology,} Springer
   (2004)

   
 \bibitem[E]{ekedahl} T. Ekedahl, {\em On the adic formalism},
   The Grothendieck Festschrift, Vol. II, 197–218, Progr. Math., 87,
   Birkh\"auser 1990

   
 \bibitem[SGA5]{sga5} A. Grothendieck et. al. {\em SGA5}, Springer LNM 589,
    (1977)

   
 \bibitem[HM]{hm} A. Huber, S. M\"uller-Stach, {\em Periods and Nori
   motives,} Springer (2017)



 
 
 \bibitem[I]{ivorra} F. Ivorra, {\em Perverse Nori motives,} Math. Res. Lett. 24 (2017), no. 4, 1097-1131.

   
   
 \bibitem[J]{jounalou} J.P. Jounalou, {\em Un suit exacte de Mayer-Vietoris
   en K-theorie algebrique}, Algebraic K-theory, Springer LNM 341
   (1973)

 \bibitem[KS]{ks} M Kashiwara, P. Schapira, {\em Sheaves on
     manifolds}, Springer-Verlag,  1990.
   

 
 \bibitem[S]{saito} M. Saito,  {\em Mixed Hodge modules,}
   Publ. Res. Inst. Math. Sci. 26 (1990), no. 2, 221-333. 

 \bibitem[SZ]{sz} J. Steenbrink, S. Zucker. {\em Variations of mixed
     Hodge structure I}, Invent  Math. 80 (1985)

\bibitem[W]{weibel} C. Weibel, {\em An introduction to homological
    algebra, }Cambridge University Press, 1994

\bibitem[Z]{zucker} S. Zucker, {\em Hodge theory with degenerating coefficients. }
Ann. of Math. (2) 109 (1979), no. 3, 415–476

 \end{thebibliography}
\end{document}